\documentclass[12pt]{amsart}

\usepackage{ucs}

\usepackage{amssymb}
\usepackage{amsthm}
\usepackage{amsmath}
\usepackage{latexsym}
\usepackage[cp1251]{inputenc}
\usepackage{graphicx}
\usepackage{wrapfig}
\usepackage{caption}
\usepackage{subcaption}
\usepackage{indentfirst}
\usepackage[left=2.6cm,right=2.6cm,top=2.95cm,bottom=2.95cm,bindingoffset=0cm]{geometry}
\usepackage{enumerate}

\DeclareMathOperator{\aut}{Aut}

\DeclareMathOperator{\cay}{Cay}
\DeclareMathOperator{\cyc}{Cyc}

\DeclareMathOperator{\id}{id}

\DeclareMathOperator{\iso}{Iso}

\DeclareMathOperator{\orb}{Orb}

\DeclareMathOperator{\rk}{rk}

\DeclareMathOperator{\Span}{Span}

\DeclareMathOperator{\sym}{Sym}
\DeclareMathOperator{\rad}{rad}

\DeclareMathOperator{\alg}{Alg}

\makeatletter 
\def\@seccntformat#1{\csname the#1\endcsname. } 
\def\@biblabel#1{#1.}

\makeatother

\title{Separability of Schur rings over abelian groups of odd order}

\author{Grigory Ryabov}
\address{Sobolev Institute of Mathematics, Novosibirsk, Russia}
\address{Novosibirsk State University, Novosibirsk, Russia}
\email{gric2ryabov@gmail.com}
\thanks{The work is supported by the Russian Foundation for Basic Research (project 18-31-00051)}

\date{}

\newtheorem{prop}{Proposition}[section]

\newtheorem*{theo1}{Main Theorem}

\newtheorem{lemm}[prop]{Lemma}

\newtheorem*{corl1}{Corollary 1}
\newtheorem*{corl2}{Corollary 2}
\theoremstyle{definition}

\begin{document}

\vspace{\baselineskip}
\vspace{\baselineskip}

\vspace{\baselineskip}

\vspace{\baselineskip}

\begin{abstract}
An $S$-ring (a Schur ring) is said to be \emph{separable} with respect to a class of groups $\mathcal{K}$ if every algebraic isomorphism from the $S$-ring in question to an $S$-ring over a group from $\mathcal{K}$ is induced by a combinatorial isomorphism. A finite group $G$ is said to be \emph{separable} with respect to $\mathcal{K}$ if every $S$-ring over $G$ is separable with respect to $\mathcal{K}$. We prove that every abelian group $G$ of order $9p$, where $p$ is a prime, is separable with respect to the class of all finite abelian groups. Modulo previously obtained results, this completes a classification of noncyclic abelian groups of odd order that are separable with respect to the class of all finite abelian groups. Also this implies that the Weisfeiler-Leman dimension of the class of Cayley graphs over~$G$ is at most~2.  
\end{abstract}

\maketitle

\section{Introduction}

A \emph{Schur ring} or \emph{$S$-ring} over a finite group $G$ can be defined as a subring of the group ring $\mathbb{Z}G$ that is a free $\mathbb{Z}$-module spanned by a partition of $G$ closed under
taking inverse and containing the identity element $e$ of $G$ as a class (see Section~2 for the exact definition). The elements of the partition are called \emph{the basic sets} of the $S$-ring. The first construction of such ring was proposed by Schur~\cite{Schur}. The general theory of $S$-rings was developed by Wielandt in~\cite{Wi}. Schur and Wielandt used $S$-rings to study permutation groups containing regular subgroups. Concerning the theory of $S$-rings, we refer the reader to~\cite{CP,MP}.

Let $\mathcal{A}$ and $\mathcal{A}^{'}$ be $S$-rings over groups $G$ and $G^{'}$ respectively. \emph{A (combinatorial) isomorphism} from $\mathcal{A}$ to $\mathcal{A}^{'}$  is defined to be a bijection  $f:G\rightarrow G^{'}$ satisfying the following condition: for every basic set $X$ of $\mathcal{A}$ there exists a basic set $X^{'}$ of $\mathcal{A}^{'}$ such that $f$ is an isomorphism of the Cayley graphs $\cay(G,X)$ and $\cay(G^{'},X^{'})$. \emph{An algebraic isomorphism} from  $\mathcal{A}$ to $\mathcal{A}^{'}$ is defined to be a bijection from the set of basic sets of $\mathcal{A}$ to the set of basic sets of $\mathcal{A}^{'}$ that preserves structure constants. Every algebraic isomorphism is extended by linearity to the ring isomorphism from $\mathcal{A}$ to $\mathcal{A}^{'}$. One can verify that every combinatorial isomorphism induces the algebraic one. However, the converse statement is not true in general (see~\cite{EP1}).

Let $\mathcal{K}$ be a class of groups. Following~\cite{EP3}, we say that an $S$-ring $\mathcal{A}$ is \emph{separable} with respect to $\mathcal{K}$ if every algebraic isomorphism from $\mathcal{A}$ to an $S$-ring over a group from $\mathcal{K}$ is induced by a combinatorial isomorphism. An importance of separable $S$-rings comes from the observation that a separable $S$-ring is determined up to isomorphism by the tensor of its structure constants (with respect to the basis corresponding to the partition of the underlying group).

The following definition was suggested in~\cite{Ry2}: a finite group $G$ is said to be \emph{separable} with respect to $\mathcal{K}$ if every $S$-ring over $G$ is separable with respect to $\mathcal{K}$. The classes of all finite cyclic groups and all finite abelian groups are denoted by $\mathcal{K}_C$ and $\mathcal{K}_A$ respectively. The question whether a given group is separable with respect to some class is quite complicated. Even among cyclic groups there are infinitely many both separable and nonseparable with respect to $\mathcal{K}_C$ groups (see~\cite{EP1,EP5}). 

We say that a finite group $G$ is \emph{weakly separable} if it is separable with respect to the class of groups isomorphic to $G$. The cyclic and elementary abelian groups of order~$n$ are denoted by $C_n$ and $E_n$ respectively. In~\cite{Ry4} it was proved that every weakly separable abelian group belongs to one of several explicitly given families. From the results obtained in~\cite{Ry1,Ry2, Ry3} it follows that some of these families, namely the groups $C_{p^k}$, $C_{4p}$, $E_4\times C_p$, and $C_q\times C_{q^k}$, where $p$ is a prime, $q\in\{2,3\}$, and $k\geq 1$, are separable with respect to $\mathcal{K}_A$. However, for other families the question whether the groups from these families are separable with respect to $\mathcal{K}_A$ remains open. In this paper we give an affirmative answer to this question for two more families, namely for $C_{9p}$ and $E_9 \times C_p$, where $p$ is a prime. The main result of the paper can be formulated as follows. 

\begin{theo1}\label{main}
An abelian group of order $9p$ is separable with respect to $\mathcal{K}_A$ for every prime $p$.
\end{theo1}

As an immediate consequence of the Main Theorem,~\cite[Theorem~1]{Ry1}, and~\cite[Theorem~1.2, Theorem~1.3]{Ry4}, we obtain a classification of noncyclic abelian groups of odd order that are separable with respect to $\mathcal{K}_A$.

\begin{corl1}\label{oddorder}
A noncyclic abelian group of odd order is separable with respect to $\mathcal{K}_A$ if and only if it is isomorphic to $C_3\times C_{3^k}$ for an integer $k\geq 1$, or $E_9\times C_p$  for a prime $p\geq 3$.
\end{corl1}

One more motivation to study separable $S$-rings comes from the Cayley graph isomorphism problem. If a group $G$ is separable with respect to a class $\mathcal{K}$ then given a Cayley graph over $G$ and a Cayley graph over an arbitrary group from $\mathcal{K}$ one can test efficiently whether these two Cayley graphs are isomorphic by using the Weisfeiler-Leman algorithm~\cite{WeisL}. In the sense of~\cite{KPS} this means that the Weisfeiler-Leman dimension of the class of Cayley graphs over $G$ is at most~2. 

\begin{corl2}
Let $p$ be a prime, $G$ an abelian group of order $9p$, and $\mathcal{G}$ the class of Cayley graphs over~$G$. Then the Weisfeiler-Leman dimension of $\mathcal{G}$ is at most~$2$. 
\end{corl2}

The text of the paper is organized in the following way. Section~2 contains a background of $S$-rings, especially, isomorphisms of $S$-rings, schurian and separable $S$-rings, wreath and tensor products of $S$-rings, $S$-rings over cyclic groups. In Section~3 we give a description and properties of $S$-rings over $C_{9p}$ and $E_9\times C_p$, where $p$ is a prime. In Section~3 we prove the Main Theorem.
\\
\\
\\
{\bf Notation.}

The set of non-identity elements of a group $G$ is denoted by  $G^\#$.

The projections of $X\subseteq A\times B$ to $A$ and $B$ are denoted by $X_A$ and $X_B$ respectively.

If $X\subseteq G$ then the element $\sum_{x\in X} {x}$ of the group ring $\mathbb{Z}G$ is denoted by $\underline{X}$.

The order of $g\in G$ is denoted by $|g|$.

The set $\{x^{-1}:x\in X\}$ is denoted by $X^{-1}$.

The subgroup of $G$ generated by $X$ is denoted by $\langle X\rangle$; we also set $\rad(X)=\{g\in G:\ gX=Xg=X\}$.

If  $m\in \mathbb{Z}$ then the set $\{x^m: x \in X\}$ is denoted by $X^{(m)}$.

If $X\subseteq G$ then the set of arcs $\{(g,xg):~x\in X,~g\in G\}$ of the Cayley digraph $\cay(G,X)$ is denoted by $R(X)$.

The group of all permutations of $G$ is denoted by $\sym(G)$.

The subgroup of $\sym(G)$ induced by right multiplications of $G$ is denoted by $G_{right}$.

For a set $\Delta\subseteq \sym(G)$ and a section $S=U/L$ of $G$ we set 
$$\Delta^S=\{f^S:~f\in \Delta,~S^f=S\},$$
where $S^f=S$ means that $f$ permutes the $L$-cosets in $U$ and $f^S$ denotes the bijection of $S$ induced by $f$.

If a group $K$ acts on a set $\Omega$ then the set of all orbits of $K$ on $\Omega$ is denoted by $\orb(K,\Omega)$.

If $K\leq \sym(\Omega)$ and $\alpha\in \Omega$ then the stabilizer of $\alpha$ in $K$ is denoted by $K_{\alpha}$.

If $H\leq G$ then the normalizer and centralizer of $H$ in $G$ are denoted by $N_G(H)$ and $C_G(H)$ respectively.

The cyclic group of order $n$ is denoted by  $C_n$.

The elementary abelian group of order $n$ is denoted by $E_n$.

The class of all finite abelian groups is denoted by $\mathcal{K}_A$.

\section{Preliminaries}
In this section we provide a background of $S$-rings. In general, we follow~\cite{Ry2}, where the most part of the material is contained.

\subsection{Definitions and basic facts}

Let $G$ be a finite group and $\mathbb{Z}G$  the integer group ring. The identity element of $G$ is denoted by~$e$.  A subring  $\mathcal{A}\subseteq \mathbb{Z} G$ is called an \emph{$S$-ring} (a \emph{Schur} ring) over $G$ if there exists a partition $\mathcal{S}=\mathcal{S}(\mathcal{A})$ of~$G$ such that:

$(1)$ $\{e\}\in\mathcal{S}$,

$(2)$  if $X\in\mathcal{S}$ then $X^{-1}\in\mathcal{S}$,

$(3)$ $\mathcal{A}=\Span_{\mathbb{Z}}\{\underline{X}:\ X\in\mathcal{S}\}$.

\noindent The elements of $\mathcal{S}$ are called the \emph{basic sets} of  $\mathcal{A}$ and the number $\rk(\mathcal{A})=|\mathcal{S}|$ is called the \emph{rank} of~$\mathcal{A}$. The $S$-ring of rank~2 over $G$ is denoted by $\tau(G)$. Denote the set $\{|X|:~X\in \mathcal{S}(\mathcal{A}), X\neq \{e\}\}$ by $\mathcal{N}(\mathcal{A})$.

Let $X,Y\in\mathcal{S}$. If $Z\in \mathcal{S}$ then the number of distinct representations of $z\in Z$ in the form $z=xy$ with $x\in X$ and $y\in Y$ does not depend on the choice of $z\in Z$. Denote this number by $c^Z_{X,Y}$. One can see that $\underline{X}~\underline{Y}=\sum_{Z\in \mathcal{S}(\mathcal{A})}c^Z_{X,Y}\underline{Z}$. Therefore the numbers  $c^Z_{X,Y}$ are the structure constants of $\mathcal{A}$ with respect to the basis $\{\underline{X}:\ X\in\mathcal{S}\}$.

A set $X \subseteq G$ is called an \emph{$\mathcal{A}$-set} if $\underline{X}\in \mathcal{A}$. A subgroup $H \leq G$ is called an \emph{$\mathcal{A}$-subgroup} if $H$ is an $\mathcal{A}$-set. One can check that for every $\mathcal{A}$-set $X$ the groups $\langle X \rangle$ and $\rad(X)$ are $\mathcal{A}$-subgroups.

Let $L \unlhd U\leq G$. A section $U/L$ is called an \emph{$\mathcal{A}$-section} if $U$ and $L$ are $\mathcal{A}$-subgroups. If $S=U/L$ is an $\mathcal{A}$-section then the module
$$\mathcal{A}_S=Span_{\mathbb{Z}}\left\{\underline{X}^{\pi}:~X\in\mathcal{S}(\mathcal{A}),~X\subseteq U\right\},$$
where $\pi:U\rightarrow U/L$ is the canonical epimorphism, is an $S$-ring over $S$.

\begin{lemm}~\cite[Lemma~2.1]{EKP} \label{intersection}
Let $\mathcal{A}$ be an $S$-ring over a group $G$, $H$ an $\mathcal{A}$-subgroup of $G$, and $X \in \mathcal{S}(\mathcal{A})$. Then the number $|X\cap Hx|$ does not depend on $x\in X$.
\end{lemm}

\subsection{Isomorphisms and schurity}

Let $\mathcal{A}$ and $\mathcal{A}^{'}$ be $S$-rings over groups $G$  and  $G^{'}$ respectively. A bijection $f$ from $G$ to $G^{'}$ is defined to be~\emph{a (combinatorial) isomorphism} from $\mathcal{A}$ to $\mathcal{A}^{'}$ if 
$$\{R(X)^f:~X\in \mathcal{S}(\mathcal{A})\}=\{R(X^{'}):~X^{'}\in \mathcal{S}(\mathcal{A}^{'})\},$$
where $R(X)^f=\{(h^f,g^f):~(h,g)\in R(X)\}$. If there exists an isomorphism from $\mathcal{A}$ to $\mathcal{A}^{'}$ then we say that $\mathcal{A}$ and $\mathcal{A}^{'}$ are \emph{isomorphic} and write $\mathcal{A}\cong \mathcal{A}^{'}$. The group $\iso(\mathcal{A})$ of all isomorphisms from $\mathcal{A}$ to itself has a normal subgroup
$$\aut(\mathcal{A})=\{f\in \iso(\mathcal{A}): R(X)^f=R(X)~\text{for every}~X\in \mathcal{S}(\mathcal{A})\}.$$
This subgroup is called the \emph{automorphism group} of $\mathcal{A}$. One can verify that $\aut(\mathcal{A})\geq G_{right}$ and $N_{\aut(\mathcal{A})}(G_{right})_e\leq \aut(G)$. The $S$-ring $\mathcal{A}$ is said to be \emph{normal} if $G_{right}\unlhd \aut(\mathcal{A})$. If $S$ is an $\mathcal{A}$-section then $\aut(\mathcal{A})^S\leq\aut(\mathcal{A}_S)$. 

Let $K$ be a subgroup of $\sym(G)$ containing $G_{right}$. Schur proved in~\cite{Schur} that the $\mathbb{Z}$-submodule
$$V(K,G)=\Span_{\mathbb{Z}}\{\underline{X}:~X\in \orb(K_e,~G)\}$$
is an $S$-ring over $G$. An $S$-ring $\mathcal{A}$ over  $G$ is called \emph{schurian} if $\mathcal{A}=V(K,G)$ for some $K\leq \sym(G)$ with $K\geq G_{right}$. Clearly, $\mathbb{Z}G=V(G_{right},G)$ and $\tau(G)=V(\sym(G),G)$. If $\mathcal{A}=V(K,G)$ for some $K\leq \sym(G)$ containing $G_{right}$ and $S$ is an $\mathcal{A}$-section then $\mathcal{A}_S=V(K^S,G)$. So if $\mathcal{A}$ is schurian then $\mathcal{A}_S$ is also schurian for every $\mathcal{A}$-section $S$. One can verify that $\mathcal{A}$ is schurian if and only if $\mathcal{A}=V(\aut(\mathcal{A}),G)$ or, equivalently, $\mathcal{S}(\mathcal{A})=\orb(\aut(\mathcal{A})_e,G)$.

Let $K \leq \aut(G)$. Then  $\orb(K,G)$ forms a partition of  $G$ that defines an  $S$-ring $\mathcal{A}$ over $G$.  In this case  $\mathcal{A}$ is called \emph{cyclotomic} and denoted by $\cyc(K,G)$. If $\mathcal{A}=\cyc(K,G)$ for some $K\leq \aut(G)$ then $\mathcal{A}=V(KG_{right},G)$. So every cyclotomic $S$-ring is schurian. If $\mathcal{A}=\cyc(K,G)$ for some $K\leq \aut(G)$ and $S$ is an $\mathcal{A}$-section then $\mathcal{A}_S=\cyc(K^S,G)$.

A \emph{Cayley isomorphism} from  $\mathcal{A}$  to $\mathcal{A}^{'}$   is defined to be a group isomorphism $f:G\rightarrow G^{'}$ such that $\mathcal{S}(\mathcal{A})^f=\mathcal{S}(\mathcal{A}^{'})$. If there exists a Cayley isomorphism from  $\mathcal{A}$  to $\mathcal{A}^{'}$ we say that $\mathcal{A}$ and $\mathcal{A}^{'}$ are \emph{Cayley isomorphic}  and write $\mathcal{A}\cong_{\cay}\mathcal{A}^{'}$. Every Cayley isomorphism is a (combinatorial) isomorphism, however the converse statement is not true.

Sets $X,Y\subseteq G$ are called \emph{rationally conjugate} if $Y=X^{(m)}$ for some $m\in \mathbb{Z}$ coprime to $|G|$.  If $G$ is abelian and $m$ is coprime to~$|G|$ then $\sigma_m$ and $\sigma_0$ denote the automorphisms of $G$ such that $g^{\sigma_m}=g^m$ and $g^{\sigma_0}=g^{-1}$  respectively for every $g\in G$. The following  statement is known as the Schur theorem on multipliers.

\begin{lemm}\cite[Theorem~23.9, (a)]{Wi}\label{burn}
Let $\mathcal{A}$ be an $S$-ring over an abelian group  $G$. Then  $X^{(m)}\in \mathcal{S}(\mathcal{A})$  for every  $X\in \mathcal{S}(\mathcal{A})$ and every  $m\in \mathbb{Z}$ coprime to $|G|$. Other words, $\sigma_m$ is a Cayley isomorphism from $\mathcal{A}$ to istself for every  $m\in \mathbb{Z}$ coprime to $|G|$.
\end{lemm}

\subsection{Algebraic isomorphisms and separability}

As in the previous subsection, $\mathcal{A}$ and $\mathcal{A}^{'}$ are $S$-rings over groups $G$  and  $G^{'}$ respectively. A bijection $\varphi:\mathcal{S}(\mathcal{A})\rightarrow\mathcal{S}(\mathcal{A}^{'})$ is defined to be an \emph{algebraic isomorphism} from $\mathcal{A}$  to $\mathcal{A}^{'}$ if 
$$c_{X,Y}^Z=c_{X^{\varphi},Y^{\varphi}}^{Z^{\varphi}}$$ 
for every $X,Y,Z\in \mathcal{S}(\mathcal{A})$. The mapping $\underline{X}\rightarrow \underline{X}^{\varphi}$ is extended by linearity to the ring isomorphism from $\mathcal{A}$  to $\mathcal{A}^{'}$. If there exists an algebraic isomorphism from  $\mathcal{A}$  to $\mathcal{A}^{'}$ we say that $\mathcal{A}$ and $\mathcal{A}^{'}$ are \emph{algebraically isomorphic} and write $\mathcal{A}\cong_{\alg}\mathcal{A}^{'}$. Every isomorphism $f$ of $S$-rings  preserves structure constants  and hence $f$ induces the algebraic isomorphism denoted by~$\varphi_f$. An $S$-ring $\mathcal{A}$ is defined to be \emph{separable} with respect to a class of groups $\mathcal{K}$ if every algebraic isomorphism from $\mathcal{A}$ to an $S$-ring over a group from $\mathcal{K}$ is induced by an isomorphism. A finite group $G$ is defined to be \emph{separable} with respect to $\mathcal{K}$ if every $S$-ring over $G$ is separable with respect to $\mathcal{K}$.

For every finite group $G$ the $S$-rings $\tau(G)$ and $\mathbb{Z}G$ are separable with respect to the class of all finite groups. In the former case there exists the unique algebraic isomorphism from the $S$-ring of rank~2 over $G$ to the $S$-ring of rank~2 over a given  group of order $|G|$ and this algebraic isomorphism is induced by every isomorphism. In the latter case every basic set is singleton and hence every algebraic isomorphism is induced by an isomorphism in a natural way.

Further everywhere throughout the text the word ``separable'' means ``separable with respect to $\mathcal{K}_A$'' and we will write ``separable'' instead of ``separable with respect to $\mathcal{K}_A$'' for short.

Let $\varphi:\mathcal{A}\rightarrow \mathcal{A}^{'}$ be an algebraic isomorphism. One can see that $\varphi$ is extended to a bijection between  $\mathcal{A}$- and $\mathcal{A}^{'}$-sets and hence between  $\mathcal{A}$- and $\mathcal{A}^{'}$-sections. The images of an $\mathcal{A}$-set $X$ and an $\mathcal{A}$-section $S$ under the action of $\varphi$ are denoted by $X^{\varphi}$ and $S^{\varphi}$ respectively. If $S$ is an $\mathcal{A}$-section then  $\varphi$ induces the algebraic isomorphism $\varphi_S:\mathcal{A}_S\rightarrow \mathcal{A}^{'}_{S^{'}}$, where $S^{'}=S^{\varphi}$. The above bijection between the $\mathcal{A}$- and $\mathcal{A}^{'}$-sets is, in fact, an isomorphism of the corresponding lattices. It follows that
$$\langle X^{\varphi} \rangle = \langle X \rangle ^{\varphi}~\text{and}~\rad(X^{\varphi})=\rad(X)^{\varphi}$$
for every  $\mathcal{A}$-set $X$. Due to $c^{\{e\}}_{X,Y}=\delta_{Y,X^{-1}}|X|$ and $|X|=c^{\{e\}}_{X,X^{-1}}$ , where $X,Y\in \mathcal{S}(\mathcal{A})$ and $\delta_{X,X^{-1}}$ is the Kronecker delta, we obtain that $(X^{-1})^{\varphi}=(X^{\varphi})^{-1}$ and $|X|=|X^{\varphi}|$ for every $\mathcal{A}$-set $X$. In particular, $|G|=|G^{'}|$.

\begin{lemm}\cite[Lemma 2.1]{EP5}\label{uniq}
Let $\mathcal{A}$ and $\mathcal{A}^{'}$ be $S$-rings over groups $G$ and $G^{'}$ respectively. Let $\mathcal{B}$ be the $S$-ring generated by $\mathcal{A}$ and an element $\xi\in \mathbb{Z}G$ and $\mathcal{B}^{'}$  the $S$-ring generated by $\mathcal{A}^{'}$ and an element $\xi^{'}\in \mathbb{Z}G^{'}$.  Then given algebraic isomorphism  $\varphi:\mathcal{A}\rightarrow \mathcal{A}^{'}$ there is at most one algebraic isomorphism $\psi:\mathcal{B}\rightarrow \mathcal{B}^{'}$ extending $\varphi$ and such that $\xi^{\psi}=\xi^{'}$.
\end{lemm}

\subsection{Wreath and tensor products}

Let $S=U/L$ be an $\mathcal{A}$-section of $G$. The $S$-ring~$\mathcal{A}$ is called the \emph{$S$-wreath product} of $\mathcal{A}_U$ and $\mathcal{A}_{G/L}$ if $L\trianglelefteq G$ and $L\leq\rad(X)$ for each basic set $X$ outside~$U$. In this case we write $\mathcal{A}=\mathcal{A}_U\wr_{S}\mathcal{A}_{G/L}$. If $U=L$ then  $\mathcal{A}$ coincides with the \emph{wreath product} of $\mathcal{A}_L$ and $\mathcal{A}_{G/L}$ denoted by $\mathcal{A}_L\wr\mathcal{A}_{G/L}$. The $S$-wreath product is called \emph{nontrivial} or \emph{proper}  if $L\neq \{e\}$ and $U\neq G$.

\begin{lemm}\cite[Lemma 4.4]{Ry1}\label{sepwr}
Let $\mathcal{A}$ be the $S=U/L$-wreath product over an abelian group $G$. Suppose that $\mathcal{A}_U$ and $\mathcal{A}_{G/L}$ are separable and $\aut(\mathcal{A}_U)^S=\aut(\mathcal{A}_{S})$. Then $\mathcal{A}$ is separable. In particular, the wreath product of two separable $S$-rings over abelian groups is separable.
\end{lemm}

Let  $\mathcal{A}_1$ and $\mathcal{A}_2$ be $S$-rings over $G_1$ and $G_2$ respectively. Then the set
$$\mathcal{S}=\mathcal{S}(\mathcal{A}_1)\times \mathcal{S}(\mathcal{A}_2)=\{X_1\times X_2:~X_1\in \mathcal{S}(\mathcal{A}_1),~X_2\in \mathcal{S}(\mathcal{A}_2)\}$$
forms a partition of  $G=G_1\times G_2$ that defines an  $S$-ring over $G$. This $S$-ring is called the  \emph{tensor product}  of $\mathcal{A}_1$ and $\mathcal{A}_2$ and denoted by $\mathcal{A}_1 \otimes \mathcal{A}_2$.

\begin{lemm}\cite[Lemma 2.3]{EKP}\label{tenspr}
Let $\mathcal{A}$ be an $S$-ring over an abelian group $G=G_1\times G_2$. Suppose that $G_1$ and $G_2$ are $\mathcal{A}$-subgroups. Then 

$(1)$ $X_{G_i}\in \mathcal{S}(\mathcal{A})$ for all $X\in \mathcal{S}(\mathcal{A})$ and $i=1,2;$

$(2)$ $\mathcal{A} \geq \mathcal{A}_{G_1}\otimes \mathcal{A}_{G_2}$, and the equality is attained whenever $\mathcal{A}_{G_i}=\mathbb{Z}G_i$ for some $i\in \{1,2\}$.
\end{lemm}

\begin{lemm} \label{septens}
The	 tensor product of two separable $S$-rings is separable.
\end{lemm}
	
\begin{proof}
As it was noted in~\cite{Ry2}, the statement of the lemma follows from~\cite[Theorem~1.20]{E}.
\end{proof}

\subsection{$S$-rings over cyclic groups}

Let $G$ be a cyclic group and $\mathcal{A}$ an $S$-ring over $G$. Put $\rad(\mathcal{A})=\rad(X)$, where $X$ is a basic set of $\mathcal{A}$ containing a generator of $G$. Note that $\rad(\mathcal{A})$ does not depend on the choice of $X$. Indeed, if $Y\in \mathcal{S}(\mathcal{A})$, $\langle Y \rangle=G$, and $Y\neq X$ then  $X$ and $Y$ are rationally conjugate by Lemma~\ref{burn} and hence $\rad(X)=\rad(Y)$.

\begin{lemm}\label{circ}
Let $\mathcal{A}$ be an $S$-ring over a cyclic group $G$. Then one of the following statements holds:

$(1)$ $\rk(\mathcal{A})=2$;

$(2)$ $\mathcal{A}$ is the tensor product of two $S$-rings over proper subgroups of $G$;

$(3)$ $\mathcal{A}$ is the nontrivial $S$-wreath product for some $\mathcal{A}$-section $S$;

$(4)$ $\mathcal{A}$ is a normal cyclotomic $S$-ring with $|\rad(\mathcal{A})|=1$.
\end{lemm}

\begin{proof}
The statement of the lemma follows from~\cite[Theorem~4.1, Theorem~4.2]{EP4}.
\end{proof}

Let $\Omega$ be a finite set. Permutation groups $K,~K^{'}\leq \sym(\Omega)$   are called  2-\emph{equivalent} if  $\orb(K,\Omega^2)=\orb(K^{'},\Omega^2)$. A permutation group  $K\leq \sym(\Omega)$ is called 2-\emph{isolated} if  it is the only group which is 2-equivalent to $K$.

\begin{lemm}\label{2isol}
Let $\mathcal{A}$ be an $S$-ring over a group $G$ of prime order $p$. Suppose that $p\leq 3$ or $\rk(\mathcal{A})>2$. Then $\aut(\mathcal{A})$ is 2-isolated. 
\end{lemm}

\begin{proof}
If $p\leq 3$ then $|\aut(\mathcal{A})|\leq 6$ and the statement of the lemma can be checked by the straightforward computation. Suppose that $\rk(\mathcal{A})>2$. Then Statement~1 of Lemma~\ref{circ} does not hold for $\mathcal{A}$. Statements~2 and~3 of Lemma~\ref{circ} also do not hold for $\mathcal{A}$ because $G$ is of prime order. Therefore Statement~4 of Lemma~\ref{circ} holds for $\mathcal{A}$, i.e. $\mathcal{A}$ is normal and cyclotomic. So $\mathcal{A}$ is schurian and $\aut(\mathcal{A})_e\leq \aut(G)$. 

Let $X\in \mathcal{S}(\mathcal{A})$ and $X\neq \{e\}$. Then $X\in \orb(\aut(\mathcal{A})_e,G)$ because $\mathcal{A}$ is schurian. Suppose that $f\in \aut(\mathcal{A})_e\leq \aut(G)$ and $x^f=x$ for some $x\in X$. Then $f$ is trivial because $x$ is a generator of $G$. So $X$ is a regular orbit of $\aut(\mathcal{A})_e$. The group $\aut(G)$ is cyclic. So both of the groups $\aut(\mathcal{A})_e$ and $\aut(\mathcal{A})_e^X$ are cyclic groups of order $|X|$. Thus, $X$ is a faithful regular orbit of $\aut(\mathcal{A})_e$ and $\aut(\mathcal{A})$ is 2-isolated by~\cite[Lemma~8.2]{MP2}. The lemma is proved.
\end{proof}

\begin{lemm}\cite[Lemma~6.7, (1)]{EP2}\label{cycsep}
Let $\mathcal{A}$ be a normal cyclotomic $S$-ring with trivial radical over a cyclic group $G$. Then every algebraic isomorphism from $\mathcal{A}$ to itself is induced by a Cayley isomorphism.
\end{lemm}

\subsection{Subdirect product}

Let $K$ and $M$  be groups. Suppose that $K_0\trianglelefteq K$, $M_0\trianglelefteq M$, and $K/K_0\cong M/M_0$. Let $\pi_1:K\rightarrow K/K_0$ and $\pi_2:M\rightarrow M/M_0$ be the canonical epimorphisms and $\psi$ the isomorphism from $K/K_0$ to $M/M_0$. We can form the subdirect product $W(K,K_0,M,M_0,\psi)$ of $K$ and $M$ in the following way:
$$W(K,K_0,M,M_0,\psi)=\{(\alpha,\beta)\in K\times M| (\alpha^{\pi_1})^{\psi}=\beta^{\pi_2}\}.$$
We say that the subdirect product of two groups is \emph{nontrivial} if it does not coincide with the direct product of these groups.

\section{$S$-rings over an abelian group of order~$9p$}

Let $p\geq 5$ be a prime. The main goal of this section is to give a description and properties of $S$-rings over the groups $C_{9p}$ and $E_9\times C_p$. 
 
Let $K=\aut(C_3)$ and $M=\aut(C_p)$. It is easy to see that there exists the unique nontrivial subdirect product $W_0=W(K,K_0,M,M_0,\psi)$ of $K$ and $M$, where $K_0$ is the trivial subgroup of $K$, $M_0$ is the subgroup of $M$ of index~2, and $\psi$ is the unique isomorphism from $K/K_0$ to $M/M_0$. Put $\mathcal{A}_0=\cyc(W_0,C_{3p})$. 

\begin{lemm}\label{c3p}
Let $\mathcal{A}$ be a cyclotomic $S$-ring over $C_{3p}$, $\rk(\mathcal{A}_{C_p})=2$, and $\mathcal{A}\neq \mathcal{A}_{C_3}\otimes \mathcal{A}_{C_p}$. Then $\mathcal{A}=\mathcal{A}_0$.
\end{lemm}

\begin{proof}
Let $\mathcal{A}=\cyc(W,C_{3p})$ for some $W\leq \aut(C_{3p})$. Then $C_3$ and $C_p$ are $\mathcal{A}$-subgroups. Note that $\rk(\mathcal{A}_{C_3})=2$ because otherwise $\mathcal{A}_{C_3}=\mathbb{Z}C_3$ and $\mathcal{A}=\mathcal{A}_{C_3}\otimes \mathcal{A}_{C_p}$ by Statement~2 of Lemma~\ref{tenspr}, a contradiction with the assumption of the lemma. Since $\tau(C_3)=\cyc(W^{C_3},C_{3})$ and $\tau(C_p)=\cyc(W^{C_p},C_{p})$, we obtain that $W^{C_3}=\aut(C_3)$ and $W^{C_p}=\aut(C_p)$. So $W$ is the subdirect product of $\aut(C_3)$ and $\aut(C_p)$. This subdirect product is nontrivial because $\mathcal{A}\neq \mathcal{A}_{C_3}\otimes \mathcal{A}_{C_p}$. However, $W_0$ is the unique nontrivial subdirect product of $\aut(C_3)$ and $\aut(C_p)$. We conclude that $W=W_0$ and hence $\mathcal{A}=\mathcal{A}_0$. The lemma is proved.
\end{proof}

Let $G\in \{C_{9p}, E_9\times C_p\}$. Define the $S$-rings $\mathcal{A}_i^{*}$, $i\in\{1,2,3\}$, over $G$ in the following way:
$$\mathcal{A}_1^{*}=\mathcal{A}_0\wr_{C_{3p}/C_3}(\mathbb{Z}C_3\otimes \tau(C_p)),$$ 
$$\mathcal{A}_2^{*}=\mathcal{A}_0\wr_{C_{3p}/C_3}(\tau(C_3)\otimes \tau(C_p)),$$
$$\mathcal{A}_3^{*}=\mathcal{A}_0\wr_{C_{3p}/C_3}\mathcal{A}_0.$$

Put $A=\langle a \rangle$, $B=\langle b \rangle$, $C=\langle c \rangle$, and $P=\langle z \rangle$,  where $|a|=|b|=3$, $|c|=9$, and $|z|=p$. Also put $E=\langle a \rangle \times \langle b \rangle$ and $C_0=\langle c_0 \rangle$, where $c_0=c^3$. These notations are valid until the end of the paper. From now throughout this section $H\in \{C,E\}$ and $G=H\times P$.  

\begin{lemm}\label{conj0}
Let $\mathcal{A}$ be an $S$-ring over $G$ and $X,Y\in \mathcal{S}(\mathcal{A})$. Suppose that $X_H=Y_H$. Then $X$ and $Y$ are rationally conjugate.
\end{lemm}

\begin{proof}
Let $h\in X_H=Y_H$, $z_1\in X_P$, and $z_2\in Y_P$. Since $P$ is of prime order~$p\geq 5$, there exists an integer $m$ coprime to~$p$ such that $z_1^m=z_2$. There exists an integer $l$ such that $l\equiv 1~\mod~9$ and $l \equiv m~\mod~p$. Then $(hz_1)^l=hz_2\in X^{(m)}\cap Y$. Lemma~\ref{burn} implies that $X^{(m)}=Y$. The lemma is proved.
\end{proof}

\begin{lemm}\label{autsection}
Let $\mathcal{A}$ be an $S$-ring over $G$ and $S=U/L$ an $\mathcal{A}$-section. Suppose that $|S|\leq 3$ or $|S|=p$ and $\rk(\mathcal{A}_S)>2$. Then $\aut(\mathcal{A}_U)^S=\aut(\mathcal{A}_{S})$.
\end{lemm}

\begin{proof}
The $S$-ring $\mathcal{A}$ is schurian by~\cite[Theorem~1.1]{EP4} if $H=C$ and by~\cite[Theorem~1.1]{PR} if $H=E$. This implies that $\mathcal{A}_U$ and $\mathcal{A}_S$ are also schurian. So the groups $\aut(\mathcal{A}_U)^S$ and $\aut(\mathcal{A}_S)$ are 2-equivalent because $\orb(\aut(\mathcal{A}_U)^S,S^2)=\orb(\aut(\mathcal{A}_S),S^2)=\{R(X):~X\in \mathcal{S}(\mathcal{A}_S)\}$. The group $\aut(\mathcal{A}_S)$ is 2-isolated by Lemma~\ref{2isol}. Therefore $\aut(\mathcal{A}_U)^S=\aut(\mathcal{A}_S)$. The lemma is proved.
\end{proof}

If $\mathcal{A}$ is an $S$-ring over $G\cong E_9\times C_p$ then by $\rad(\mathcal{A})$ we mean the group generated by $\rad(X)$, where $X$ runs over all basic sets of $\mathcal{A}$ containing an element of order $3p$.

\begin{lemm}\label{9p}
Let $\mathcal{A}$ be an $S$-ring over $G$. Then one of the following statements holds:

$(1)$ $\rk(\mathcal{A})=2$;

$(2)$ $\mathcal{A}$ is the tensor product of two $S$-rings over proper nontrivial subgroups of $G$;

$(3)$ $\mathcal{A}$ is the nontrivial $S$-wreath product for some $\mathcal{A}$-section $S=U/L$ and $\aut(\mathcal{A}_U)^S=\aut(\mathcal{A}_{S})$;

$(4)$ $\mathcal{A}\cong \mathcal{A}_i^{*}$ for some $i\in\{1,2,3\}$; 

$(5)$ $H=C$ and $\mathcal{A}$ is a normal cyclotomic $S$-ring with $|\rad(\mathcal{A})|=1$;

$(6)$ $H=E$, $|\rad(\mathcal{A})|=1$, and $\mathcal{A}\cong_{\cay}\cyc(W,G)$, where $W=W(K,K_0,M,M_0,\psi)$ for some $K_0\trianglelefteq K\leq \aut(E)$ from Table~1, some $M_0\trianglelefteq M\leq \aut(P)$ with $M/M_0\cong K/K_0$, and some isomorphism $\psi$ from $K/K_0$ to $M/M_0$.
\end{lemm}

\begin{table}
{\small
\begin{tabular}{|l|l|l|l|l|l|}
  \hline
 no. & $K$ & generators of $K$ & $K_0$ & generators of $K_0$ & $|K:K_0|$ \\
  \hline
  
1. & $C_2$ & $(a,b)\rightarrow (a^2,b^2)$ & trivial & $(a,b)\rightarrow (a,b)$ & 2 \\ \hline
	
2. & $E_4$ & $(a,b)\rightarrow (a^2,b),~(a,b)\rightarrow (a,b^2)$ & $C_2$ & $(a,b)\rightarrow (a^2,b^2)$ & 2 \\ \hline	
	
3. & $C_3$ & $(a,b)\rightarrow (a,ab)$ & trivial & $(a,b)\rightarrow (a,b)$ & 3 \\ \hline	

4. & $C_6$ & $(a,b)\rightarrow (a^2,ab^2)$ & $C_2$ & $(a,b)\rightarrow (a^2,b^2)$ & 3 \\ \hline

5. & $C_6$ & $(a,b)\rightarrow (a^2,ab^2)$ & trivial & $(a,b)\rightarrow (a,b)$ & 6 \\ \hline

6. & $D_8$ & $(a,b)\rightarrow (b^2,a),~(a,b)\rightarrow (b,a)$ & $C_2\times C_2$ & $(a,b)\rightarrow (a^2,b^2),~(a,b)\rightarrow (a^2,b)$ & 2 \\ \hline

7. & $C_4$ & $(a,b)\rightarrow (b^2,a)$ & $C_2$ & $(a,b)\rightarrow (a^2,b^2)$ & 2 \\ \hline

8. & $C_4$ & $(a,b)\rightarrow (b^2,a)$ & trivial & $(a,b)\rightarrow (a,b)$ & 4 \\ \hline

9. & $C_8$ & $(a,b)\rightarrow (ab,a^2b)$ & $C_4$ & $(a,b)\rightarrow (b^2,a)$ & 2 \\ \hline

10. & $C_8$ & $(a,b)\rightarrow (ab,a^2b)$ & $C_2$ & $(a,b)\rightarrow (a^2,b^2)$ & 4 \\ \hline

11. & $C_8$ & $(a,b)\rightarrow (ab,a^2b)$ & trivial & $(a,b)\rightarrow (a,b)$ & 8 \\ \hline

\end{tabular}
}
\caption{}

\end{table}

\begin{proof}
Firstly suppose that $H=C$. Then from Lemma~\ref{circ} it follows that one of the Statements~1,2,5 of the lemma holds for $\mathcal{A}$ or $\mathcal{A}$ is the nontrivial $S$-wreath product for some $\mathcal{A}$-section $S=U/L$. Consider the latter case. Clearly, $|S|\in\{1,3,p\}$. If $|S|\leq 3$ or $|S|=p$ and $\rk(\mathcal{A}_S)>2$ then $\aut(\mathcal{A}_U)^S=\aut(\mathcal{A}_S)$ by Lemma~\ref{autsection} and Statement~3 of the lemma holds. 

Suppose that $|S|=p$ and $\rk(\mathcal{A}_S)=2$. Note that in this case
$$L=C_0~\text{and}~U=C_0\times P$$ 
because $C_0$ is the unique subgroup of $G$ of order~3. The group $U$ is cyclic and $\rk(\mathcal{A}_U)>2$. So Lemma~\ref{circ} implies that $\mathcal{A}_U=\mathcal{A}_{C_0}\otimes \mathcal{A}_P$, or $\mathcal{A}_U=\mathcal{A}_{C_0}\wr \mathcal{A}_{U/C_0}$, or  $\mathcal{A}_U$ is cyclotomic. In the first case we have
$$\aut(\mathcal{A}_U)^S=(\aut(\mathcal{A}_{C_0})\times \sym(P))^S=\sym(S)=\aut(\mathcal{A}_S)$$
and Statement~3 of the lemma holds. In the second case we conclude that $\mathcal{A}=\mathcal{A}_{C_0}\wr \mathcal{A}_{G/C_0}$ and hence Statement~3 of the lemma holds for the section $S^{'}=C_0/C_0$. Therefore we may assume that $\mathcal{A}_U$ is cyclotomic and $\mathcal{A}_U\neq\mathcal{A}_{C_0}\otimes \mathcal{A}_P$. Now Lemma~\ref{c3p} yields that 
$$\mathcal{A}_U=\mathcal{A}_0.$$

Assume that $C$ is not an $\mathcal{A}$-subgroup. Then $C_0$ is a maximal $\mathcal{A}$-subgroup inside $C$. From~\cite[Lemma~6.2]{EKP} it follows that $\mathcal{A}=\mathcal{A}_{C_0}\wr \mathcal{A}_{G/C_0}$, or $\mathcal{A}=\mathcal{A}_{U}\wr\mathcal{A}_{G/U}$, or $\mathcal{A}=\mathcal{A}_{U}\wr_{U/P}\mathcal{A}_{G/P}$. In the first, second, and third cases $\mathcal{A}$ is the nontrivial $S^{'}=U^{'}/L^{'}$-wreath product with $|S^{'}|\leq 3$ for the sections $S^{'}=C_0/C_0$, $S^{'}=U/U$, and $S^{'}=U/P$ respectively. Due to Lemma~\ref{autsection}, we obtain that $\aut(\mathcal{A}_{U^{'}})^{S^{'}}=\aut(\mathcal{A}_{S^{'}})$ and Statement~3 of the lemma holds. Therefore we may assume that $C$ is an $\mathcal{A}$-subgroup.

Let $\pi:G\rightarrow G/C_0$ be the canonical epimorphism. The group $S=P^{\pi}$ is an $\mathcal{A}_{G/C_0}$-subgroup. The group $C^{\pi}$ is also an $\mathcal{A}_{G/C_0}$-subgroup because $C$ is an $\mathcal{A}$-subgroup. So $\mathcal{A}_{G/C_0}$ is not a wreath product of two $S$-rings. Since the group $G/C_0$ is cyclic, $\mathcal{A}_{G/C_0}=\mathcal{A}_{C^{\pi}}\otimes \mathcal{A}_{P^{\pi}}$ or $\mathcal{A}_{G/C_0}$ is cyclotomic by Lemma~\ref{circ}. Note that $\mathcal{A}_{P^{\pi}}=\tau(P^{\pi})$ because $\rk(\mathcal{A}_S)=2$. Suppose that $\mathcal{A}_{G/C_0}=\mathcal{A}_{C^{\pi}}\otimes \mathcal{A}_{P^{\pi}}$. Due to~$|C^{\pi}|=3$, we conclude that 
$$\mathcal{A}_{C^{\pi}}=\mathbb{Z}C^{\pi}~\text{or}~\mathcal{A}_{C^{\pi}}=\tau(C^{\pi}).$$ 
In the former case $\mathcal{A}\cong \mathcal{A}_1^{*}$; in the latter case $\mathcal{A}\cong\mathcal{A}_2^{*}$. We obtain that Statement~4 of the lemma holds. So we may assume that $\mathcal{A}_{G/C_0}$ is cyclotomic and $\mathcal{A}_{G/C_0}\neq\mathcal{A}_{C^{\pi}}\otimes \mathcal{A}_{P^{\pi}}$. Then $\mathcal{A}_{G/C_0}\cong \mathcal{A}_0$ by Lemma~\ref{c3p} and hence $\mathcal{A}\cong \mathcal{A}_3^{*}$. Again, Statement~4 of the lemma holds.

Now let $H=E$. From~\cite[Theorem~6.1]{PR} it follows that one of the Statements~1,2 of the lemma  holds for $\mathcal{A}$, or $\mathcal{A}$ is the nontrivial $S$-wreath product for some $\mathcal{A}$-section $S=U/L$ with $|S|\leq 3$, or $\mathcal{A}$ is cyclotomic. In the second case $\aut(\mathcal{A}_U)^S=\aut(\mathcal{A}_S)$ by Lemma~\ref{autsection} and Statement~3 of the lemma holds. 

Suppose that $\mathcal{A}$ is cyclotomic. All cyclotomic $S$-rings over $E\times P$ are described in~\cite[Section~5]{PR}. It can be verified by inspecting the above $S$-rings one after the other that one of the Statements~2,4,6 of the lemma holds or $\mathcal{A}$ is the nontrivial $S$-wreath product for some $\mathcal{A}$-section $S=U/L$ with $|S|=p$ and $\rk(\mathcal{A}_S)>2$. In the latter case $\aut(\mathcal{A}_U)^S=\aut(\mathcal{A}_{S})$ by Lemma~\ref{autsection} and Statement~3 of the lemma holds. The lemma is proved.
\end{proof}

Note that Statement~3 of Lemma~\ref{9p} does not hold for $\mathcal{A}_i^{*}$, $i\in\{1,2,3\}$.

Let $H=E$. Assume that $K$ and $K_0$ are from Line~$i$ of Table~1, $i\in\{1,\ldots,11\}$. Then $K/K_0$ is cyclic. Given $M\leq \aut(P)$ with $|M|$ divisible by $|K:K_0|$ there exists the unique subgroup $M_0$ of $M$ with $M/M_0\cong K/K_0$ because $M$ is cyclic. The direct check yields that for every two isomorphisms $\psi_1$ and $\psi_2$ from $K/K_0$ to $M/M_0$ the $S$-rings $\cyc(W(K,K_0,M,M_0,\psi_1),G)$ and $\cyc(W(K,K_0,M,M_0,\psi_2),G)$ are Cayley isomorphic. Fix some isomorphism $\psi_0$ from $K/K_0$ to $M/M_0$ and put 
$$\mathcal{A}_i(M)=\cyc(W(K,K_0,M,M_0,\psi_0),G).$$

Suppose that Statement~6 of Lemma~\ref{9p} holds for an $S$-ring $\mathcal{A}$. Then the above discussion implies that $\mathcal{A}\cong_{\cay}\mathcal{A}_i(M)$ for some $i\in\{1,\ldots,11\}$ and some $M\leq \aut(P)$ with $|M|$ divisible by $|K:K_0|$, where  $K$ and $K_0$ are from Line~$i$ of Table~1. Set $|M|=k$. Clearly, $k\leq p-1$ and $\mathcal{N}(\mathcal{A}_P)=\{k\}$. The properties of $\mathcal{A}_{i}^{*}$ and $\mathcal{A}_i(M)$ presented in Table~2  follow directly from the definitions of $\mathcal{A}_{i}^{*}$ and $\mathcal{A}_i(M)$. In the first column of Table~2 we write ``$\mathcal{A}_{i}$" instead of ``$\mathcal{A}_i(M)$" for short.

\begin{table}

{\small
\begin{tabular}{|l|l|}
  \hline
  $\mathcal{A}$ & $\mathcal{N}(\mathcal{A})\setminus\mathcal{N}(\mathcal{A}_P)$       \\
  \hline
	
	$\mathcal{A}_{1}^{*}$ & $\{2,3,p-1,3(p-1)\}$  \\ \hline
	
	$\mathcal{A}_{2}^{*}$ & $\{2,6,p-1,6(p-1)\}$  \\ \hline
	
	$\mathcal{A}_{3}^{*}$ & $\{2,6,p-1, 3(p-1)\}$  \\ \hline
	
	$\mathcal{A}_{1}$ & $\{2,k\}$  \\ \hline

	$\mathcal{A}_{2}$ & $\{2,4,2k\}$ \\ \hline
	
	$\mathcal{A}_{3}$ & $\{1,3,k\}$  \\ \hline

  $\mathcal{A}_{4}$ & $\{2,6,2k\}$  \\ \hline

  $\mathcal{A}_{5}$ & $\{2,6,k\}$  \\ \hline

  $\mathcal{A}_{6}$ & $\{4,2k,4k\}$ \\ \hline

  $\mathcal{A}_{7}$ & $\{4,2k\}$  \\ \hline

  $\mathcal{A}_{8}$ & $\{4,k\}$  \\ \hline

  $\mathcal{A}_{9}$ & $\{8,4k\}$ \\ \hline

  $\mathcal{A}_{10}$ & $\{8,2k\}$  \\ \hline

  $\mathcal{A}_{11}$ & $\{8,k\}$  \\ \hline

\end{tabular}
}

\caption{}
\end{table}

If given $i\in\{1,\ldots,11\}$ and $M\leq \aut(P)$ we can form $\mathcal{A}_i(M)$, i.e. $|M|$ is divisible by $|K:K_0|$, where $K$ and $K_0$ are from Line~$i$ of Table~1, then we say that $\mathcal{A}_i(M)$ is \emph{well-defined}.

Let  $\mathcal{A}$ and $\mathcal{A}^{'}$ be $S$-rings over groups $G$ and $G^{'}$ respectively, where $G^{'}=H^{'}\times P$ and $H^{'}\in \{C,E\}$. Suppose that $P$ is an $\mathcal{A}$-,$\mathcal{A}^{'}$-subgroup and $\mathcal{A}\cong_{\alg}\mathcal{A}^{'}$. Since $P$ is the unique subgroup of order~$p$ in $G$ and in $G^{'}$, we conclude that $\mathcal{A}_P\cong_{\alg}\mathcal{A}^{'}_P$. The properties of an algebraic isomorphism imply that $\mathcal{N}(\mathcal{A})=\mathcal{N}(\mathcal{A}^{'})$, $\mathcal{N}(\mathcal{A}_P)=\mathcal{N}(\mathcal{A}^{'}_P)$, and hence 
$$\mathcal{N}(\mathcal{A})\setminus\mathcal{N}(\mathcal{A}_P)=\mathcal{N}(\mathcal{A}^{'})\setminus\mathcal{N}(\mathcal{A}^{'}_P).$$ 
So Statements~1-3 of the next lemma directly follow from the information presented in Table~2 and Statement~4 of the next lemma follows from the observation that $\mathcal{N}((\mathcal{A}_i(M))_P)=\{|M|\}$.

\begin{lemm}\label{nonisom}
The following statements hold:

$(1)$ Let $i,j\in\{1,2,3\}$. Then $\mathcal{A}^{*}_i\ncong_{\alg} \mathcal{A}^{*}_j$ whenever $i\neq j$;

$(2)$ Let $i\in\{1,2,3\}$, $j\in\{1,\ldots,11\}$, and $M\leq \aut(P)$ such that $\mathcal{A}_j(M)$ is well-defined. Then $\mathcal{A}^{*}_i\ncong_{\alg} \mathcal{A}_j(M)$;

$(3)$ Let $i,j\in\{1,\ldots,11\}$ and $M\leq \aut(P)$ such that $\mathcal{A}_i(M)$ and $\mathcal{A}_j(M)$ are well-defined. Then  $\mathcal{A}_i(M)\ncong_{\alg}\mathcal{A}_j(M)$ whenever $i\neq j$.

$(4)$ Let $i,j\in\{1,\ldots,11\}$ and  $M_1,M_2\leq \aut(P)$ such that $\mathcal{A}_i(M_1)$ and $\mathcal{A}_j(M_2)$ are well-defined. Then $\mathcal{A}_i(M_1)\ncong_{\alg}\mathcal{A}_j(M_2)$ whenever $M_1\neq M_2$;

\end{lemm}

\begin{lemm}\label{conj}
Let $\mathcal{A}\cong_{\cay}\mathcal{A}_i(M)$ for some $i\in\{1,\ldots,11\}\setminus\{1,7,8\}$ and $M\leq\aut(P)$ such that $\mathcal{A}_i(M)$ is well-defined. Suppose that $X,Y\in\mathcal{S}(\mathcal{A})$, $\langle X \rangle =\langle Y \rangle=G$, and $|X|=|Y|$. Then $X$ and $Y$ are rationally conjugate.
\end{lemm}

\begin{proof}
The inspecting of basic sets of $\mathcal{A}_i(M)$ for given $M\leq \aut(P)$ and every $i\in\{1,\ldots,11\}\setminus\{1,7,8\}$ implies that $X_E=Y_E$ and we are done by Lemma~\ref{conj0}.
\end{proof}

Note that Lemma~\ref{conj} does not hold for $\mathcal{A}_i(M)$, where $i\in\{1,7,8\}$.

Let $G=E\times P$ and $\mathcal{A}=\mathcal{A}_i(M)$ for some $i\in\{1,\ldots,11\}$ and some $M\leq \aut(P)$ such that $\mathcal{A}_i(M)$ is well-defined. If $i\notin\{1,7\}$ then it can be verified by inspecting basic sets of $\mathcal{A}_i(M)$ that there exists $X\in \mathcal{S}(\mathcal{A})$ satisfying the following conditions: 

(C1) $\langle X \rangle=G$;

(C2) $X\neq G^{\#}$;

(C3) $X\neq X_U\times X_V$ for every subgroups $U$ and $V$ of $G$ with $U\times V=G$;

(C4) a radical of every subset of $X$ is trivial;

(C5) if an $\mathcal{A}$-set $Y$ such that $Y=X^f$ for some $f\in \aut(G)$ is also an $\mathcal{A}_j(M)$-set for some $j\in\{1,\ldots,11\}$ such that $\mathcal{A}_j(M)$ is well-defined then $\mathcal{A}\leq \mathcal{A}_j(M)$.

Suppose that $i\in\{1,7\}$. In this case $|M|$ is even because $|M|$ is divisible by $|K:K_0|=2$. Let $Z$ be an orbit of $M$ and $Z_1,Z_2\subseteq Z$ the orbits of the subgroup of $M$ of index~2. Put 
$$X_0=\{a,a^{-1}\},~X_1=aZ_1\cup a^{-1}Z_2,~X_2=bZ_1\cup b^{-1}Z_2$$
if $i=1$ and
$$X_0=\{a,a^{-1},b,b^{-1}\},~X_1=\{a,a^{-1}\}Z_1\cup \{b,b^{-1}\}Z_2,~X_2=\{ab,a^{-1}b^{-1}\}Z_1\cup \{a^{-1}b,ab^{-1}\}Z_2$$
if $i=7$. 
Note that $X_0,X_1,X_2\in \mathcal{S}(\mathcal{A})$. Put $X=X_0\cup X_1 \cup X_2$. It is easy to see that $X$ satisfies (C1)-(C4). The fact that $X$ satisfies (C5) can be verified by inspecting of basic sets of $\mathcal{A}_j(M)$ for every $j\in\{1,\ldots,11\}$.

\begin{lemm}\label{generate}
In the above notations, $\mathcal{A}=\langle \underline{Y} \rangle$ for every $\mathcal{A}$-set $Y$ such that $Y=X^f$ for some $f\in \aut(G)$.
\end{lemm}

\begin{proof}
Put $\mathcal{B}=\langle \underline{Y} \rangle$. Let us prove that $\mathcal{B}=\mathcal{A}$. Clearly, $Y$ is a $\mathcal{B}$-set. Since $X$ satisfies (C1)-(C5), $Y$ also satisfies (C1)-(C5). So Statements~1-4 of Lemma~\ref{9p} do no hold for $\mathcal{B}$. Therefore Statement~6 of Lemma~\ref{9p} holds for $\mathcal{B}$, i.e. 
$$\mathcal{B}\cong_{\cay}\mathcal{A}_j(M^{'})$$ 
for some $j\in\{1,\ldots,11\}$ and $M^{'}\leq\aut(P)$ such that $\mathcal{A}_j(M^{'})$ is well-defined. In this case $E$ and $P$ are $\mathcal{B}$-subgroups. 

Statement~1 of Lemma~\ref{tenspr} implies that $Y_P$ is a $\mathcal{B}$-set. Since $\mathcal{B}$ is the minimal $S$-ring such that $Y$ is a $\mathcal{B}$-set, we conclude that 
$$Y_P\in \mathcal{S}(\mathcal{B}_P).$$
So $|M^{'}|=|Y_P|=|M|$. Since the group $\aut(P)$ is cyclic, $M^{'}=M$. Now from (C5) it follows that $\mathcal{B}\geq \mathcal{A}$. On the other hand, $\mathcal{B}\leq \mathcal{A}$ because $Y$ is an $\mathcal{A}$-set. Thus, $\mathcal{B}=\mathcal{A}$. The lemma is proved.
\end{proof}

\section{Proof of the Main Theorem}

Let $p$ be a prime and $G$ an abelian group of order $9p$. Let us prove that $G$ is separable. We start the proof with the following lemma which implies that every proper section of $G$ is separable.

\begin{lemm}\label{subgroup}
The groups $E_9$, $C_p$, and $C_{3p}$, where $p$ is a prime, are separable.
\end{lemm}

\begin{proof}
The groups $E_9$, $C_p$, and $C_9$ are separable by~\cite[Theorem 1]{Ry1}, \cite[Theorem 1.3]{EP5}, and \cite[Lemma 5.5]{Ry1} respectively. Suppose that $p\neq 3$ and $\mathcal{A}$ is an $S$-ring over $U\cong C_{3p}$. Then Lemma~\ref{circ} implies that $\rk(\mathcal{A})=2$, or $\mathcal{A}$ is the tensor product or wreath product of two $S$-rings over groups of orders~3 and $p$, or $\mathcal{A}$ is a normal cyclotomic $S$-ring with trivial radical. In the first case, obviously, $\mathcal{A}$ is separable. In the second case $\mathcal{A}$ is separable by Lemma~\ref{septens} or Lemma~\ref{sepwr}. 

Let $\mathcal{A}$ be a normal cyclotomic $S$-ring with trivial radical and $\varphi$ an algebraic isomorphism from $\mathcal{A}$ to an $S$-ring $\mathcal{A}^{'}$ over an abelian group $U^{'}$. Note that $U^{'}\cong C_{3p}$ because $|U^{'}|=|U|=3p$ and $p\neq 3$. So we may assume that $U^{'}=U$. Then from~\cite[Theorem~1.1]{M} it follows that $\mathcal{A}^{'}=\mathcal{A}$. Lemma~\ref{cycsep} yields that $\varphi$ is induced by a Cayley isomorphism. So $\mathcal{A}$ is separable. Thus, every $S$-ring over $U$ is separable and hence $U$ is separable. The lemma is proved.
\end{proof}

Let $\mathcal{A}$ be an $S$-ring over $G$. Prove that $\mathcal{A}$ is separable. Suppose that $p=2$. Then from computer calculations made by using the package COCO2P~\cite{GAP} it follows that one of the following statements holds for $\mathcal{A}$: (1) $\rk(\mathcal{A})=2$; (2) $\mathcal{A}$ is the tensor product of two $S$-rings over proper subgroups of $G$; (3) $\mathcal{A}$ is the nontrivial $S$-wreath product for some $\mathcal{A}$-section $S=U/L$ with $|S|\leq 3$. In the first case, obviously, $\mathcal{A}$ is separable. In the second case $\mathcal{A}$ is separable by Lemma~\ref{subgroup} and Lemma~\ref{septens}. In the third case $\aut(\mathcal{A}_U)^S=\aut(\mathcal{A}_S)$ by Lemma~\ref{autsection}. So $\mathcal{A}$ is separable by Lemma~\ref{subgroup} and Lemma~\ref{sepwr}. 

Let $p=3$. Then $G\cong C_{27}$, or $G\cong C_3 \times C_9$, or $G\cong E_{27}$. In the first case $\mathcal{A}$ is separable by \cite[Lemma 5.5]{Ry1}. In the second case $\mathcal{A}$ is separable by \cite[Theorem 1]{Ry1}. In the third case $\mathcal{A}$ is separable by~\cite[Theorem~1.3]{Ry3}.

Now let $p\geq 5$. Then $G=H\times P$, where $H\in\{C,E\}$, and one of the statements of Lemma~\ref{9p} holds for $\mathcal{A}$. If Statement~1 of Lemma~\ref{9p} holds for $\mathcal{A}$ then, obviously, $\mathcal{A}$ is separable. If Statement~2 of Lemma~\ref{9p} holds for $\mathcal{A}$ then $\mathcal{A}$ is separable by Lemma~\ref{subgroup} and Lemma~\ref{septens}. If Statement~3 of Lemma~\ref{9p} holds for $\mathcal{A}$ then $\mathcal{A}$ is separable by Lemma~\ref{subgroup} and Lemma~\ref{sepwr}. Therefore we may assume that one of the Statements~4-6 of Lemma~\ref{9p} holds for $\mathcal{A}$.

Let $\varphi$ be an algebraic isomorphism from $\mathcal{A}$ to an $S$-ring $\mathcal{A}^{'}$ over an abelian group~$G^{'}$ of order~$9p$. Let us prove that $\varphi$ is induced by an isomorphism. We may assume that one of the Statements~4-6 of Lemma~\ref{9p} holds for $\mathcal{A}^{'}$ because otherwise $\mathcal{A}^{'}$ is separable by the previous paragraph and hence $\varphi^{-1}$ is induced by an isomorphism. This impies that $\varphi$ is also induced by an isomorphism and we are done.

\begin{lemm}\label{groupiso}
Under the above assumptions, if one of the Statements~5,6 of Lemma~\ref{9p} holds for $\mathcal{A}$ then $G\cong G^{'}$.
\end{lemm}

\begin{proof}
The radical of $\mathcal{A}$ is trivial. So the radical of $\mathcal{A}^{'}$ is also trivial by the properties of an algebraic isomorphism. This means that one of the Statements~5,6 of Lemma~\ref{9p} holds for $\mathcal{A}^{'}$.

Assume the contrary that $G\ncong G^{'}$. Without loss of generality let $G=C\times P$ and $G^{'}=E\times P$. Then $\mathcal{A}$ is a cyclotomic normal $S$-ring with trivial radical and $\mathcal{A}^{'}\cong_{\cay}\mathcal{A}_i(M)$ for some $i\in\{1,\ldots,11\}$ and some $M\leq \aut(P)$ such that $\mathcal{A}_i(M)$ is well-defined. Note that $C$ and $P$ are the unique $\mathcal{A}$-subgroups of orders~9 and~$p$ respectively and $E$ and $P$ are the unique $\mathcal{A}^{'}$-subgroups of orders~9 and~$p$ respectively. So the properties of an algebraic isomorphism yield that 
$$C^{\varphi}=E~\text{and}~P^{\varphi}=P.~\eqno(1)$$

Let $X$ be a basic set of $\mathcal{A}$ outside $C\cup P$ with $\langle X \rangle=G$. Then $X^{'}=X^{\varphi}$ is a basic set of $\mathcal{A}^{'}$ outside $E\cup P$ with $\langle X^{'} \rangle=G^{'}$ by the properties of an algebraic isomorphism. Put $Y=X_C$, $Z=X_P$, $Y^{'}=X^{'}_E$, and $Z^{'}=X^{'}_P$. Due to Statement~1 of Lemma~\ref{tenspr}, we obtain that $Y\in\mathcal{S}(\mathcal{A}_C)$, $Z\in \mathcal{S}(\mathcal{A}_P)$, $Y^{'}\in\mathcal{S}(\mathcal{A}_E)$, and $Z^{'}\in \mathcal{S}(\mathcal{A}^{'}_P)$. 

Clearly, $c_{Y,Z}^X=1$. Since $\varphi$ is an algebraic isomorphism, we conclude that $c_{Y^{\varphi},Z^{\varphi}}^{X^{'}}=1$. From~(1) it follows that $Y^{\varphi}\subseteq E$ and $Z^{\varphi}\subseteq P$. The sets $Y^{'}$ and $Z^{'}$ are the unique basic sets of $\mathcal{A}^{'}_E$ and $\mathcal{A}^{'}_P$ respectively with $c_{Y^{'},Z^{'}}^{X^{'}}=1$. Therefore 
$$Y^{\varphi}=Y^{'}~\text{and}~Z^{\varphi}=Z^{'}.$$

Since $\langle X \rangle=G$ and $\mathcal{A}$ is cyclotomic, the set $X$ consists of elements of order~$9p$. So the element $\underline{Z}$ enters the element $\underline{X}^3$ with a coefficient $m$ divisible by~3. The number $l=|x^{'}E\cap X^{'}|$ does not depend on the choice of $x^{'}\in X^{'}$ by Lemma~\ref{intersection}. The set  $X^{'}$ consists of elements of order~$3p$ because $\langle X^{'} \rangle=G^{'}$ and $\mathcal{A}^{'}$ is cyclotomic. So the element $\underline{Z}^{'}$ enters the element $(\underline{X}^{'})^3$ with a coefficient $m^{'}$ such that $m^{'}\equiv l~\mod~3$. The properties of an algebraic isomorphism imply that $m=m^{'}$ because  $X^{\varphi}=X^{'}$ and $Z^{\varphi}=Z^{'}$. So $l$ is divisible by~3. From the definition of $\mathcal{A}_i(M)$ it follows that $l=|K_0|$, where $K_0\leq \aut(E)$ is from Line~$i$ of Table~1. However, $|K_0|$ is not divisible by~3 for every $K_0$ from Table~1, a contradiction. The lemma is proved. 
\end{proof}

\begin{lemm}\label{sringiso}
Under the above assumptions, $\mathcal{A}\cong \mathcal{A}^{'}$.
\end{lemm}

\begin{proof}
If  Statement~5 of Lemma~\ref{9p} holds for $\mathcal{A}$ then $G\cong G^{'}\cong C_{9p}$ by Lemma~\ref{groupiso}. Now the statement of the lemma follows from~\cite[Theorem~1.1]{M}.

Suppose that Statement~6 of Lemma~\ref{9p} holds for $\mathcal{A}$, i.e. $\mathcal{A}\cong_{\cay} \mathcal{A}_i(M)$ for some $i\in\{1,\ldots,11\}$ and $M\leq \aut(P)$ such that $\mathcal{A}_i(M)$ is well-defined. Lemma~\ref{groupiso} implies that $G\cong G^{'}\cong E_9\times C_p$. Since $|\rad(\mathcal{A})|=1$, we have $|\rad(\mathcal{A}^{'})|=1$. So Statement~6 of Lemma~\ref{9p} holds for $\mathcal{A}^{'}$, i.e. $\mathcal{A}^{'}\cong_{\cay} \mathcal{A}_j(M^{'})$ for some $j\in\{1,\ldots,11\}$ and $M^{'}\leq \aut(P)$ such that $\mathcal{A}_j(M^{'})$ is well-defined. Lemma~\ref{nonisom} yields that $\mathcal{A}_i(M)\cong_{\alg} \mathcal{A}_j(M^{'})$ if and only if $i=j$ and $M=M^{'}$. Therefore $\mathcal{A}\cong \mathcal{A}^{'}$.

Suppose that Statement~4 of Lemma~\ref{9p} holds for $\mathcal{A}$, i.e $\mathcal{A}\cong \mathcal{A}_i^{*}$ for some $i\in\{1,2,3\}$. Note that $|\rad(\mathcal{A}^{'})|>1$ because $|\rad(\mathcal{A})|>1$. So Statement~4 of Lemma~\ref{9p} holds for $\mathcal{A}^{'}$, i.e. $\mathcal{A}^{'}\cong \mathcal{A}_j^{*}$ for some $i\in\{1,2,3\}$. The $S$-rings $\mathcal{A}_i^{*}$ and $\mathcal{A}_j^{*}$ are algebraically isomorphic if and only if $i=j$ by Lemma~\ref{nonisom}. Thus, $\mathcal{A}\cong \mathcal{A}^{'}$. The lemma is proved.  
\end{proof}

In view of Lemma~\ref{sringiso}, we may assume that $\mathcal{A}=\mathcal{A}^{'}$. If Statement~5 of Lemma~\ref{9p} holds for $\mathcal{A}$ then $\varphi$ is induced by a Cayley isomorphism by Lemma~\ref{cycsep}. Suppose that $\mathcal{A}\cong_{\cay}\mathcal{A}_i(M)$ for some $i\in\{1,\ldots,11\}\setminus\{1,7,8\}$ and $M\leq\aut(P)$ such that $\mathcal{A}_i(M)$ is well-defined. Then there exists $X\in \mathcal{S}(\mathcal{A})$ satisfying (C1)-(C5) and $\mathcal{A}=\langle \underline{X} \rangle$ by Lemma~\ref{generate}. The properties of an algebraic isomorphism imply that $|X^{\varphi}|=|X|$ and $\langle X^{\varphi} \rangle =G$. So $X$ and $X^{\varphi}$ are rationally conjugate by Lemma~\ref{conj}. This means that $X^{\varphi}=X^{\sigma_m}$ for some $m$ coprime to~$9p$. Due to Lemma~\ref{generate}, we have that
$$\mathcal{A}=\langle \underline{X} \rangle=\langle \underline{X}^{\varphi} \rangle.$$
Now $\varphi=\varphi_{\sigma_m}$ by Lemma~\ref{uniq} and we are done.  

It remains to consider the following situations:

$(1)$ $\mathcal{A}\cong \mathcal{A}_i^{*}$ for some $i\in\{1,2,3\}$;

$(2)$ $\mathcal{A}\cong_{\cay}\mathcal{A}_i(M)$ for some $i\in\{1,7,8\}$ and $M\leq\aut(P)$ such that $\mathcal{A}_i(M)$ is well-defined. 

In the first situation without loss of generality we may assume that $G=E\times P$ and $A$ is an $\mathcal{A}$-subgroup of order~3. Denote the orbits of $M\leq \aut(P)$ with $|\aut(P):M|=2$ by $P_1$ and $P_2$.

\noindent \textbf{Case 1: $\mathcal{A}\cong\mathcal{A}_1^{*}$.} In this case the basic sets of $\mathcal{A}$ are the following:
$$X_0=\{e\},~X_1=\{a,a^{-1}\},~X_2=bA,~X_3=b^{-1}A,$$
$$Y_0=P^{\#},~Y_1=aP_1\cup a^{-1}P_2,~Y_2=aP_2\cup a^{-1}P_1,~Y_3=bAP^{\#},~Y_4=b^{-1}AP^{\#}.$$
Since $\varphi$ preserves cardinalities of basic sets and maps a given $\mathcal{A}$-subgroup to an $\mathcal{A}$-subgroup of the same order, we conclude that 
$$X_0^{\varphi}=X_0,~X_1^{\varphi}=X_1,~Y_0^{\varphi}=Y_0,$$ 
and each of the sets $\{X_2,X_3\}$, $\{Y_1,Y_2\}$, $\{Y_3,Y_4\}$ is invariant under $\varphi$. We may assume that 
$$X_2^{\varphi}=X_2~\text{and}~X_3^{\varphi}=X_3.$$ 
Indeed, otherwise replace $\varphi$ by $\varphi_1=\varphi \sigma_0$. Then $X_2^{\varphi_1}=X_2$, $X_3^{\varphi_1}=X_3$, and $\varphi$ is induced by an isomorphism if and only if $\varphi_1$ is induced by an isomorphism. 

One can see that $c_{X_2,Y_0}^{Y_3}=1$. So $c_{X_2,Y_0}^{Y_3^{\varphi}}=c_{X_2^{\varphi},Y_0^{\varphi}}^{Y_3^{\varphi}}=1$. Similarly, $c_{X_3,Y_0}^{Y_4}=1$ and hence $c_{X_3,Y_0}^{Y_4^{\varphi}}=1$. Therefore
$$Y_3^{\varphi}=Y_3~\text{and}~Y_4^{\varphi}=Y_4.$$
If $Y_1^{\varphi}=Y_1$ and $Y_2^{\varphi}=Y_2$ then $\varphi$ is trivial and it is induced by the identity isomorphism. If $Y_1^{\varphi}=Y_2$ and $Y_2^{\varphi}=Y_1$ then $\varphi$ is induced by a Cayley isomorphism $f_0\in \aut(G)$ from $\mathcal{A}$ to itself such that 
$$a^{f_0}=a^{-1},~b^{f_0}=b~\text{and}~f_0^P=\id_P.$$

\noindent\textbf{Case 2: $\mathcal{A}\cong\mathcal{A}_2^{*}$.} In this case the basic sets of $\mathcal{A}$ are the following:
$$X_0=\{e\},~X_1=\{a,a^{-1}\},~X_2=bA\cup b^{-1}A,$$
$$Y_0=P^{\#},~Y_1=aP_1\cup a^{-1}P_2,~Y_2=aP_2\cup a^{-1}P_1,~Y_3=bAP^{\#}\cup b^{-1}AP^{\#}.$$
Due to the properties of an algebraic isomorphism, we obtain that
$$X_0^{\varphi}=X_0,~X_1^{\varphi}=X_1,~X_2^{\varphi}=X_2,~Y_0^{\varphi}=Y_0,~Y_3^{\varphi}=Y_3.$$ 
Therefore $\varphi$ is trivial or $\varphi$ interchanges $Y_1$ and $Y_2$. In the first case $\varphi$ is  induced by the identity isomorphism; in the second case $\varphi$ is induced by $f_0$.

\noindent\textbf{Case 3: $\mathcal{A}\cong\mathcal{A}_3^{*}$.} In this case the basic sets of $\mathcal{A}$ are the following:
$$X_0=\{e\},~X_1=\{a,a^{-1}\},~X_2=bA\cup b^{-1}A,$$
$$Y_0=P^{\#},~Y_1=aP_1\cup a^{-1}P_2,~Y_2=aP_1\cup a^{-1}P_2,~Y_3=bAP_1\cup b^{-1}AP_2,~Y_4=bAP_2\cup b^{-1}AP_1.$$
The properties of an algebraic isomorphism imply that
$$X_0^{\varphi}=X_0,~X_1^{\varphi}=X_1,~X_2^{\varphi}=X_2,~Y_0^{\varphi}=Y_0,$$
and each of the sets  $\{Y_1,Y_2\}$, $\{Y_3,Y_4\}$ is invariant under $\varphi$. We may assume that 
$$Y_3^{\varphi}=Y_3~\text{and}~Y_4^{\varphi}=Y_4.$$ 
Indeed, otherwise replace $\varphi$ by $\varphi_1=\varphi \sigma_m$, where $m$ is an integer coprime to~$9p$ such that $Y_4^{(m)}=Y_3$ (such $m$ exists by Lemma~\ref{conj0}). Then $\varphi$ is induced by an isomorphism if and only if $\varphi_1$ is induced by an isomorphism.

Again, we have that $\varphi$ is trivial or $\varphi$ interchanges $Y_1$ and $Y_2$ and hence  $\varphi$ is  induced by the identity isomorphism or by~$f_0$.

In Cases~4-6 we assume that $\mathcal{A}=\mathcal{A}_i(M)$ for some $i\in\{1,7,8\}$ and $M\leq\aut(P)$ such that $\mathcal{A}_i(M)$ is well-defined. Since in these cases $|K:K_0|$ is even, we conclude that $|M|$ is even. Denote the unique subgroup of $M$ of index~2 by~$M_0$. Let $Z\in \mathcal{S}(\mathcal{A}_P)=\orb(M,P)$ and $Z_1,Z_2$ the orbits of $M_0$ inside $Z$.

\noindent\textbf{Case 4: $\mathcal{A}=\mathcal{A}_1(M)$.} In this case  the sets
$$X_0=\{a,a^{-1}\},~X_1=aZ_1 \cup a^{-1}Z_2,~X_2=bZ_1\cup b^{-1}Z_2$$
are basic sets of $\mathcal{A}$. Put $X=X_0\cup X_1\cup X_2$. Since $\varphi$ preserves cardinalities of basic sets, we conclude that $X_0^{'}=X_0^{\varphi}=\{h,h^{-1}\}$ for some $h\in E$. We may assume that $Z^{\varphi}=Z$. Indeed, otherwise replace $\varphi$ by $\varphi_1=\varphi \sigma_m$, where $m$ is an integer coprime to~$9p$ such that $(Z^{\varphi})^{(m)}=Z$. Then $\varphi$ is induced by an isomorphism if and only if $\varphi_1$ is induced by an isomorphism.

Every basic set of $\mathcal{A}$ outside $E\cup P$ is of the form $tS_1\cup t^{-1}S_2$, where $t\in E$ and $S_1,S_2$ are the orbits of $M_0$. Note that $c_{X_0,Z}^{X_1}=1$ and hence $c_{X_0^{'},Z}^{X_1^{'}}=1$, where $X_1^{'}=X_1^{\varphi}$. So without loss of generality we may assume that $X_1^{'}=hZ_1\cup h^{-1}Z_2$. The set $X_2^{'}=X_2^{\varphi}$ lies outside $E\cup P$. Also $(X_2^{'})_P=Z$ because $(X_2)_P=Z$ and $Z^{\varphi}=Z$. In view of $c_{X_0,Z}^{X_2}=0$, we have that $c_{X_0^{'},Z^{'}}^{X_2^{'}}=0$. Therefore $X_2^{'}=uZ_1\cup u^{-1}Z_2$ for some $u\in E\setminus \langle h \rangle$. Thus,
$$X^{'}=X^{\varphi}=X_0^{'}\cup X_1^{'}\cup X_2^{'}=\{h,h^{-1}\}\cup hZ_1\cup h^{-1}Z_2 \cup uZ_1\cup u^{-1}Z_2.$$ 

Note that $\mathcal{A}=\cyc(W,G)$, where $W=W(K,K_0,M,M_0,\psi_0)$, $K_0$ is trivial, and $K=\langle \sigma_0\rangle$. Let us define $f\in \aut(G)$ in the following way:
$$a^f=h,~b^f=u,~f^P=\id_P.$$
Since $K=\langle \sigma_0 \rangle\leq Z(\aut(G))$, we obtain that $f\in C_{\aut(G)}(K\times M)\leq C_{\aut(G)}(W)$. So $f$ is a Cayley isomorphism from $\mathcal{A}$ to itself. One can see that $X^f=X^{'}$. Lemma~\ref{generate} implies that $\mathcal{A}=\langle \underline{X} \rangle=\langle \underline{X}^{'} \rangle$. Now Lemma~\ref{uniq} yields that $\varphi=\varphi_f$.

In Cases~5-6 basic sets of $\mathcal{A}_E$ are the following:
$$\{e\},~\{a,a^{-1},b,b^{-1}\},~\{ab,a^{-1}b^{-1},ab^{-1},a^{-1}b\}.~\eqno(2)$$

\noindent\textbf{Case 5: $\mathcal{A}=\mathcal{A}_7(M)$.} In this case  the sets
$$X_0=\{a,a^{-1},b,b^{-1}\},~X_1=\{a,a^{-1}\}Z_1 \cup \{b,b^{-1}\}Z_2,~X_2=\{ab,a^{-1}b^{-1}\}Z_1\cup \{ab^{-1},a^{-1}b\}Z_2$$
are basic sets of $\mathcal{A}$. Put $X=X_0\cup X_1\cup X_2$. Note that $X_0^{'}=X_0^{\varphi}=\{h,h^{-1},u,u^{-1}\}$ for some $(h,u)\in\{(a,b),(ab,a^{-1}b)\}$ because $\varphi$ preserves cardinalities of basic sets. As in Case~4, we may assume that $Z^{\varphi}=Z$. 

Every basic set of $\mathcal{A}$ outside $E\cup P$ is of the form $\{t,t^{-1}\}S_1\cup \{r,r^{-1}\}S_2$, where $(t,r)\in \{(a,b),(ab,a^{-1}b)\}$ and $S_1,S_2$ are the orbits of $M_0$. Since $c_{X_0,Z}^{X_1}=1$, we obtain that $c_{X_0^{'},Z}^{X_1^{'}}=1$, where $X_1^{'}=X_1^{\varphi}$. So without loss of generality we may assume that $X_1^{'}=\{h,h^{-1}\}Z_1\cup \{u,u^{-1}\}Z_2$. Put $X_2^{'}=X_2^{\varphi}$. The set $X_2^{'}$ lies outside $E\cup P$. Also $(X_2^{'})_P=Z$ because $(X_2)_P=Z$ and $Z^{\varphi}=Z$. Due to $c_{X_0,Z}^{X_2}=0$, we have that $c_{X_0^{'},Z}^{X_2^{'}}=0$. This yields that $(X_2^{'})_E\neq X_0^{'}$ and hence $(X_2^{'})_E=E\setminus (X_0^{'}\cup \{e\})=\{hu,h^{-1}u^{-1},h^{-1}u,hu^{-1}\}$. Therefore 
$$X_2^{'}=\{hu,h^{-1}u^{-1}\}Z_1 \cup \{hu^{-1},h^{-1}u\}Z_2~\text{or}~X_2^{'}=\{hu^{-1},h^{-1}u\}Z_1 \cup \{hu,h^{-1}u^{-1}\}Z_2.$$ 
Thus,
$$X^{'}=X^{\varphi}=\{h,h^{-1},u,u^{-1}\}\cup \{h,h^{-1}\}Z_1\cup \{u,u^{-1}\}Z_2 \cup \{hu,h^{-1}u^{-1}\}Z_1\cup \{hu^{-1},h^{-1}u\}Z_2$$ 
or
$$X^{'}=X^{\varphi}=\{h,h^{-1},u,u^{-1}\}\cup \{h,h^{-1}\}Z_1\cup \{u,u^{-1}\}Z_2 \cup \{hu^{-1},h^{-1}u\}Z_1\cup \{hu,h^{-1}u^{-1}\}Z_2.$$ 
In the first case define $f\in\aut(G)$ as follows:
$$a^f=h,~b^f=u,~f^P=\id_P;$$
in the second case define $f\in\aut(G)$ as follows:
$$a^f=h^{-1},~b^f=u,~f^P=\id_P.$$
One can see that $X^f=X^{'}$. By the definition, $\mathcal{A}=\cyc(W,G)$ for $W=W(K,K_0,M,M_0,\psi_0)$, where $K$ is generated by $\sigma:(a,b)\rightarrow (b^{-1},a)$ and $K_0=\langle \sigma_0 \rangle$. The straightforward check with using $(h,u)\in\{(a,b),(ab,a^{-1}b)\}$ implies that $f^E\in N_{\aut(G)}(K)$ in both cases and hence $f\in N_{\aut(G)}(W)$. Therefore $f$ is a Cayley isomorphism from $\mathcal{A}$ to itself. From Lemma~\ref{generate} it follows that $\mathcal{A}=\langle \underline{X} \rangle=\langle \underline{X}^{'} \rangle$. Thus, $\varphi=\varphi_f$ by Lemma~\ref{uniq}.

\noindent\textbf{Case 6: $\mathcal{A}=\mathcal{A}_8(M)$.} In this case there exists $X\in \mathcal{S}(\mathcal{A})$ satisfying~(C1)-(C5). So $\langle \underline{X} \rangle=\mathcal{A}$ by Lemma~\ref{generate}. If $X^{\varphi}=X^{(m)}$ for some integer $m$ coprime to~$9p$ then $\langle \underline{X}^{\varphi} \rangle=\mathcal{A}$ by Lemma~\ref{generate}. Therefore $\varphi=\varphi_{\sigma_m}$ by Lemma~\ref{uniq}.

Suppose that $X^{\varphi}$ and $X$ are not rationally conjugate. Then $X_E\neq X^{\varphi}_E$ by Lemma~\ref{conj0}. Statement~1 of Lemma~\ref{tenspr} implies that $X_E, X^{\varphi}_E\in \mathcal{S}(\mathcal{A}_E)$ and $X_P, X^{\varphi}_P\in \mathcal{S}(\mathcal{A}_P)$. In view of~(2), we may assume that 
$$X_E=\{a,a^{-1},b,b^{-1}\}~\text{and}~X^{\varphi}_E=\{ab,a^{-1}b^{-1},ab^{-1},a^{-1}b\}.$$

Note that $\mathcal{A}=\cyc(W,G)$, where $W=W(K,K_0,M,M_0,\psi_0)$, $K_0$ is trivial, and $K=\langle \sigma \rangle$, where $\sigma=(a,b)\rightarrow (b^2,a)$. Define $f\in \aut(G)$ as follows:
$$a^f=ba,~b^f=ba^2,~f^P=\id_P.$$
The straightforward check yields that $\sigma f^E=f^E\sigma$ and hence $f\in C_{\aut(G)}(K\times M)\leq  C_{\aut(G)}(W)$. Therefore $f$ is a Cayley isomorphism from $\mathcal{A}$ to itself. From the definition of $f$ it follows that $X^f_E=X^{\varphi}_E$. So $X^{f\sigma_m}=X^{\varphi}$ for some integer $m$ coprime to~$9p$ by Lemma~\ref{conj0}. Due to Lemma~\ref{generate}, we have that  $\langle \underline{X}^{\varphi} \rangle=\langle \underline{X}^{f\sigma_m}\rangle=\mathcal{A}$. Lemma~\ref{uniq} implies that $\varphi=\varphi_{f\sigma_m}$. 

We checked that in all cases $\varphi$ is induced by an isomorphism and hence $\mathcal{A}$ is separable. The theorem is proved.

\end{document}